\documentclass[12pt]{amsart}

\usepackage{amsmath, amscd, amssymb, latexsym, amsthm, xspace, color}%,a4}
\usepackage{ifpdf}
\usepackage{tikz}
\usetikzlibrary{matrix,arrows,shapes,decorations.pathmorphing,decorations.markings,decorations.pathreplacing}
\usepackage{graphicx}
\usepackage[all]{xy}
\usepackage{mathtools}
\newlength{\halfbls}
\ifpdf
\else
\usepackage{epsfig}
\fi

\def\be{\begin{equation}}   \def\ee{\end{equation}}     \def\bes{\begin{equation*}}    \def\ees{\end{equation*}}
\def\ba{\be\begin{aligned}} \def\ea{\end{aligned}\ee}   \def\bas{\bes\begin{aligned}}  \def\eas{\end{aligned}\ees}
\def\={\;=\;}  \def\+{\,+\,} 
\def\mat#1#2#3#4{\begin{pmatrix}#1&#2\\#3&#4\\ \end{pmatrix}}

\newcommand{\CC}{\mathbb{C}}

\newcommand{\HH}{\mathbb{H}}

\newcommand{\PP}{\mathbb{P}}
\newcommand{\QQ}{\mathbb{Q}}
\newcommand{\RR}{\mathbb{R}}

\newcommand{\ZZ}{\mathbb{Z}}

\newcommand{\cA}{{\mathcal A}}
\newcommand{\cJ}{{\mathcal J}}

\newcommand{\cE}{{\mathcal E}}

\newcommand{\cH}{{\mathcal H}}

\newcommand{\cM}{{\mathcal M}}

\newcommand{\cX}{{\mathcal X}}

\newcommand{\tr}{{\rm tr}}

\DeclareMathOperator{\Sp}{Sp}

\newcommand{\ol}{\overline}

\newcommand{\GL}{{\rm GL}}

\newcommand{\Ker}{{\rm Ker}}

\newcommand{\Jac}{{\rm Jac} \,}

\newcommand{\Prym}{{\rm Prym}}

\newcommand{\moduli}[1][g]{{\cM}_{#1}}
\newcommand{\AVmoduli}[1][g]{{\cA}_{#1}}

\theoremstyle{definition}
\newtheorem{Defi}{Definition}%[section]
\newtheorem{Rem}[Defi]{Remark}

\theoremstyle{plain}
\newtheorem{Prop}[Defi]{Proposition}

\newtheorem{Cor}[Defi]{Corollary}
\newtheorem{Thm}[Defi]{Theorem}

\title[Shimura curves in genus~4]{Explicit formulas for infinitely many Shimura curves in genus~4}
\author{Samuel Grushevsky}
\address{Mathematics Department, Stony Brook University, Stony Brook, NY 11794-3651, USA.}
\email{sam@math.sunysb.edu}
\thanks{Research of the first author is supported in part by National Science Foundation under the grants DMS-1201369 and DMS-1501265. Research of the second author is supported in part by ERC-StG-257137.}
\author{Martin M\"oller}
\address{Institut f\"ur Mathematik, Goethe-Universit\"at Frankfurt, Robert-Mayer-Str. 6-8
60325 Frankfurt am Main, Germany}
\email{moeller@math.uni-frankfurt.de}
\begin{document}
\begin{abstract}
In this paper we construct infinitely many Shimura curves
contained in the locus of Jacobians of genus four curves. All Jacobians
in these families are $\ZZ/3$ covers of varying elliptic curves that appear
in a geometric construction of Pirola, and include an example of a Shimura-Teichm\"uller
curve that parameteri\-zes Jacobians that are suitable $\ZZ/6$ covers
of $\PP^1$. We compute explicitly the period matrices of the Shimura curves
we construct using the original construction of Shimura for moduli spaces
of abelian varieties with automorphisms.
\end{abstract}

\maketitle

\section*{Introduction}
A well-known question in relating the geometric and arithmetic properties of symmetric domains is the following problem expectation of Oort (posed in \cite{oortconj}, see \cite{moor}): for $g$ sufficiently large, does the moduli space $\cA_g$ of complex principally polarized abelian varieties contain any Shimura subvariety contained generically in the Torelli locus~$\cJ_g$ --- the locus of Jacobians of smooth genus $g$ curves?  A finite number of examples of Shimura curves contained in the Torelli locus have been constructed by de Jong, Moonen, Mumford, Noot, Oort, and others, most of them arising as Galois covers of $\PP^1$ with branch points varying --- see the survey \cite{moor} for the history of the problem, details, and references; see also \cite{fred2} for more examples arising as non-abelian Galois covers.

From the other direction, a lot of work has been devoted to proving that for $g$ sufficiently large the Torelli locus contains no special subvarieties. The latest progress in this direction has been made by Chen, Lu and Zuo, who proved that there do not exist Shimura curves of Mumford type, or of maximal variation, or contained in the locus of hyperelliptic Jacobians, for some explicit low bounds on $g$, see \cite{luzuo,chenluzuo}.

Our modest contribution in this note is an explicit construction --- by writing down families of period matrices --- of infinitely many Shimura curves contained in $\cJ_4$.
\begin{Thm}\label{thm:main}
For a dense set of points $z_1,z_2$ in a 2-dimensional complex ball, the one-parameter family of $4\times 8$ period matrices, given by
$$(Z_1\,|\, Z_2)   = \left(\smallmatrix 3\tau & 0 &\ 3\tau +3 & 0 \\
0 &\ Z_1^{(3)}(z_1,z_2) & 0 &\ Z_2^{(3)}(z_1,z_2) \endsmallmatrix\right)B^{-1}$$
where $\tau\in\HH$ is the parameter for the family, where~$B$ is given by~\eqref{eq:STbasechange}, and $Z_1^{(3)},Z_2^{(3)}$ are given by Proposition \ref{prop:genus3},  defines a Shimura curve contained generically in the locus of Jacobians of smooth curves of genus~4.
\end{Thm}
Our examples include the Shimura-Teichm\"uller curve of genus~4
(\cite{moellerST})) obtained for the values
$$z_1 \= \tfrac{-2\zeta^3 + \zeta^2 + \zeta -3}2, \qquad z_2 \= 3^{-1/4}
\, \tfrac{\zeta^3 -2\zeta^2 +1}2.$$
These Shimura curves are not of maximal degeneration or of Mumford type in the sense of \cite{luzuo}.

Unlike our purely analytic construction of infinitely many Shimura curves contained in the locus of hyperelliptic Jacobians of genus three \cite{grmo}, the construction in this note is geometric, by using  $\ZZ/3$ Galois covers of elliptic covers, and the associated Prym map studied by Pirola (\cite{pirola}).

\medskip
We announced our results at several talks starting in February 2014, including at Oberwolfach, Paris Jussieu, and Roma Tre. At the last stage of preparing our manuscript, the preprint of Frediani, Penegini, Porru \cite{frediani} appeared, which studies much more generally under what conditions families of covers of elliptic curves may lead to Shimura curves. They independently discovered our examples, and much more, while we are able to compute the period matrices explicitly using the Shimura description.

%%%%%%%%%%
\subsection*{Notation}
%%%%%%%%%%
We denote by $\moduli$ the moduli space of smooth complex curves,
and by $\AVmoduli$ the moduli space of complex principally polarized abelian
varieties (ppav). We denote by $\Jac:\moduli\to\AVmoduli$ the Torelli map
that sends the curve to its Jacobian. The main question we study here is
the existence of Shimura curves that are contained generically in
$\cJ_g:=\Jac(\moduli)$. A {\em Kuga fiber space} is an inclusion $j$ of
a $\QQ$-algebraic group $G$ into $\Sp(2g,\QQ)$, such that an arithmetic
lattice $\Gamma \subset G_\RR$ maps to $\Sp(2g,\ZZ)$ and such that the
intersection of the maximal compact
subgroup $U(g) \subset \Sp(2g,\RR)$ with $j(G_\RR)$ is a maximal compact
subgroup $K_\RR$ of $G_\RR$. We identify a Kuga fiber space with the
subvariety of the moduli stack
$$ \Gamma \backslash G_\RR/K_\RR \hookrightarrow \Sp(2g,\ZZ) \backslash \Sp(2g,\RR)/\mathop{U}(g) = \AVmoduli.$$
A {\em Shimura subvariety} of $\AVmoduli$ is the image of such an inclusion
that moreover contains a CM point. One-dimensional Kuga fiber spaces
$\Gamma \backslash G_\RR/K_\RR$ are called Kuga curves, one-dimensional
Shimura varieties are called Shimura curves.
\par
We refer to \cite{moor}
for a detailed discussion of the history and importance of the problem,
its relationship with the Andr\'e-Oort and Coleman conjectures, the current
status, and further references.

\medskip
In Section~1 we review Pirola's construction and explain how it gives
examples of Shimura curves. In Section~2 we compute the period matrix of
the Shimura-Teichm\"uller curve in genus~4. explicitly. In Section~3 we use
this to compute the necessary isogeny data to compute the data
of the period matrices for all the Shimura curves that we construct ---
we hope that this section may be of independent use as a working guide
for applying Shimura's construction.

%%%%%%%%%%
\section{Pirola's Hurwitz space of cyclic triple covers of elliptic curves by genus four curves}
%%%%%%%%%%
In \cite{pirola} Pirola studied the space of genus four smooth curves $X$
that are Galois triple covers $p:X\to E$ of elliptic curves $E$, totally
ramified over three points. We denote by $\cH$ the Hurwitz space
of such covers, which thus admits a finite map  $\pi:\cH\to
\moduli[1,3]$ and a generically injective map $\phi:\cH\to\moduli[4]$.
Pirola's interest in $\cH$ is in disproving a conjecture of Xiao on the fixed parts of families of curves.
\par
Given a point $p: X \to E$ in $\cH$ we define
$$\Prym(X/E) :=  \Ker\Big(\Jac(X) \to \Jac(E)=E\Big)$$
to be the (generalized) Prym variety of the covering. The polarization
on $\Prym(X/E)$ is given by the
restriction of the principal polarization from  $\Jac(X)$, and is of
type $(1,1,3)$, so that we get the Prym map $P:\cH\to\AVmoduli[{3,(1,1,3)}]$
to the moduli space of $(1,1,3)$-polarized abelian threefolds. Using deformation
theory, Pirola showed \cite[Section~2]{pirola} that the image $P(\cH)$ is
two-dimensional. Since $\dim\cH=3$, the generic fiber of the map $P$ is one-dimensional. These fibers are the Kuga curves we are looking for.
\par
\begin{Prop}
For any $A\in P(\cH)\subset\AVmoduli[3,(1,1,3)]$ the one-parameter family of four-dimensional
Jacobians $\Jac(P^{-1}(A))$ is isogenous to the product of $A$ (as a trivial family over
a base) and the universal family of elliptic curves with a suitable level structure.
\end{Prop}
\begin{proof}
For any $(p:X\to E)\in P^{-1}(A)$ the Jacobian $\Jac(X)$ is isogenous to the
product of the fixed abelian threefold $A$ and the elliptic curve $E$, and
the elliptic curve varies as we vary along the family. Thus what remains
to be seen is that the family of elliptic curves that is thus obtained is
indeed the universal family of elliptic curves with some level, and not
some ramified cover of it.
\par
The family of curves $y^6=x(x+1)(x-t)$ studied in detail in the next section
is a member of this family for some~$A$, in fact the period matrix will
be computed below (see~\eqref{Z3special}). This family is a Shimura curve
as discussed in \cite{moellerST} and computed again below, hence for this~$A$
the family of elliptic curves is as claimed.
\par
On the other hand, the isogeny that allows to write $\Jac(X)$ as a product is
constant over all of $P(\cH)$, since the possible isogenies are countable.
This implies that the uniformization of the elliptic curve does not vary
in the family over $P(\cH)$.
\end{proof}
\par
\begin{Cor}
For any $A\in P(\cH)$, the image of the one-parameter family of curves
$\phi(P^{-1}(A))\subset\moduli[4]$ under the Torelli map $\Jac$ is a
Kuga curve in $\AVmoduli[4]$ generically contained in the locus of Jacobians
of smooth curves.
\end{Cor}
\begin{proof}
This is obvious from the preceding proposition. Alternatively, one can
invoke the criterion from \cite{movizu} as the universal family of abelian
varieties over
such a curve is globally isogenous to the product of a trivial family and
the varying family of elliptic curves which reaches the Arakelov bound.
\end{proof}
\begin{Rem}
Note that we did not make any attempt to characterize how many times the closure of $C_A$
intersects the reducible locus nor how many times $C_A$ intersects the hyperelliptic locus.
Consequently, we do not claim a priori that any of the curves $C_A$ reaches the Arakelov bound,
just that $C_A$ is the intersection of a Kuga curve in $\AVmoduli[4]$ with $\moduli[4]$.
\end{Rem}
\begin{Rem}
If $A$ is a CM point, then the Kuga curve constructed
above is in fact a Shimura curve --- indeed, just take the elliptic curve to be CM, and then the
abelian fourfold is isogenous to a product of a CM abelian threefold and a CM elliptic curve. Thus within the two-dimensional family of Kuga curves parameterized by $P(\cH)$ we have an everywhere dense collection
of Shimura curves.
\end{Rem}

%%%%%%%%%%
\section{The period matrix for the genus four Shimura-Teichm\"uller curve}
%%%%%%%%%%
We recall that there is precisely one Shimura curve in genus~4 which happens
to be also a Teichm\"uller curve (\cite{moellerST}). It parametrizes the
family of curves given by equations
$$X_t := \lbrace y^6 = x(x+1)(x-t)\rbrace, $$
where $t$ is the parameter of the family.
In this section we compute explicitly the period matrix of
$X_t$ and then compute explicitly the isogeny between $\Jac(X_t)$
and $\Prym(X/E) \times E$.
\par
\begin{Prop}
In the bases of homology and 1-forms defined below, the period matrix of $X_t$
is given by
$$(Z_1 \, |\, Z_2) = \left( \begin{smallmatrix*}[r]
\tau & \tau &0 &\ -\tau - 1 &\, \ 1 &1 &0 &-1 \\
\zeta^2 - 1 & \ \ 1 &-\zeta^2 + 1 &1\,& 1 &-\zeta^2 &\zeta^2 &\ -\zeta^2 + 1\\
\zeta^3 - \zeta & \ -\zeta^3 & \ -2\zeta^3 + 2\zeta^2 + \zeta - 1 &\ \zeta^2 - \zeta \,&1 &\ \zeta^2 - 1 &\zeta^3 - \zeta^2 - 2\zeta + 2 &\zeta^2 \\
-\zeta^3 + \zeta &\zeta^3 &2\zeta^3 + 2\zeta^2 - \zeta - 1 &\zeta^2 + \zeta \,&1 &\zeta^2 - 1 &\ -\zeta^3 - \zeta^2 + 2\zeta + 2 &\zeta^2\\
 \end{smallmatrix*} \right)$$
where $\zeta := \zeta_{12} =  e^{2\pi i/12}$ is a primitive $12$-th root of unity
and where $t = \lambda(\tau)$ for $\lambda: \HH \to \PP^1\setminus\{0,1,\infty\}$.
\end{Prop}
\par
Our strategy is similar to that used by Guardia
for the Shimura-Teichm\"uller curve in genus 3 in
\cite[Section~3 and~4]{Guardia}, except that our situation is
somewhat more involved, as  the only automorphism of a generic point of our
family is $\alpha:(x,y) \mapsto (x,\zeta^2 y)$. We thus apply the fact that
most eigenspaces of the action of the automorphism on the cohomology are
unitary local systems, allowing us to determine much of the period matrix
from suitable special values of $t$.
\begin{proof}
We decompose the cohomology $H^1(X_t,\ZZ)$ in a direct sum of
the eigenspaces for the action of~$\alpha$. The eigenvalues are powers of $\zeta^2$, i.e.~sixth
roots of unity, and their dimensions can be computed
%The dimension
%of the $i$-th eigenspace, i.e.\ for the action of $\alpha^i$ and the rank
%of the holomorphic subbundles can be calculated
by the well-known formulas for cyclic covers (see e.g.~\cite{bouwprank}).

\begin{table}
\begin{tabular}{|c|c|c|c|c|c|}
\hline&&&&&\\
i & 1 & 2 & 3 & 4 & 5 \\
[-\halfbls] &&&&&\\
\hline&&&&& \\
${\rm rank}(H^1(X_t,\ZZ))_{\zeta^{2i}}$ & 2 & 1 & 2 & 1 & 2 \\
[-\halfbls] &&&&&\\
\hline&&&&&\\
${\rm dim}(H^0(X,\Omega^1_X))_{\zeta^{2i}}$  & 0 & 0 & 1 & 1 & 2  \\
[-\halfbls] &&&&&\\
\hline
\end{tabular}
\end{table}

%\begin{figure}[h]
%$$ \begin{array}{|c|c|c|c|c|c|}
%\hline
%i & 1 & 2 & 3 & 4 & 5 \\
%\hline &&&&& \\ [-\halfbls]
%{\rm rank}(H^1(X_t,\ZZ))_{\zeta^{2i}} & 2 & 1 & 2 & 1 & 2 \\
%[-\halfbls] &&&&&\\
%\hline &&&&& \\ [-\halfbls]
%{\rm dim}(H^0(X,\Omega^1_X))_{\zeta^{2i}}  & 0 & 0 & 1 & 1 & 2  \\
%[-\halfbls] &&&&&\\
%\hline
%\end{array}
%$$
%\end{figure}
In fact, a basis of eigenforms  on $X$ is given by
$\omega_1 = dx/y^3$, $\omega_2 = dx/y^4$, $\omega_3 = dx/y^5$ , $\omega_2 = xdx/y^5$.
Consequently, all but the $(-1)$-eigenspace are unitary and thus
the corresponding lines in the period matrix do not depend on~$t$.
\par
It suffices to compute the period matrix locally near some
fixed $t_0$. We take $t_0\in\RR_{>0}$, and let $X = X_{t_0}$.  The first step
is to define a suitable basis of $H_1(X,\ZZ)$ and of $H^0(X,\Omega^1_X)$.
With the branch points aligned as $(-1,0,t_0,\infty)$ we perform the
branch cuts and draw two paths as in Figure~\ref{fig:paths}.
\begin{figure}[ht]
\begin{tikzpicture}[scale=1]%,decoration={
\tikzset{arrow data/.style 2 args={%
      decoration={%
         markings,
         mark=at position #1 with \arrow{#2}},
         postaction=decorate}
      }%

\draw (0,0) node(1){} -- (.825,0) node(2){} -- (1.65,0) node(3){} -- (2.475,0) node(4){} -- (3.3,0)
node(5){} -- (4.125,0) node(6){} -- (4.95,0) node(7){};
\fill (1) circle (1.2pt)
      (3) circle (1.2pt)
      (5) circle (1.2pt)
      (7) circle (1.2pt);

%Figure
\draw [arrow data={0.1}{stealth}, arrow data={0.4}{stealth},
       arrow data={0.6}{stealth}, arrow data={0.9}{stealth}] plot [smooth, tension=.7]
       coordinates {(2.475,0) (1.3,1.5) (.825,0) (1.3,-1.5) (2.475,0)
                    (3.65,1.5) (4.125,0) (3.65,-1.5) (2.475,0)};

% Linien, Punkte
\draw (7.25,0) node(8){} -- (8.075,0) node(9){} -- (8.9,0) node(10){} -- (9.725,0) node(11){} -- (10.55,0)
node(12){} -- (12.1,0) node(13){};
\fill (8) circle (1.2pt)
      (10) circle (1.2pt)
      (12) circle (1.2pt)
      (13) circle (1.2pt);

%Figure
\draw [arrow data={0.1}{stealth}, arrow data={0.4}{stealth},
       arrow data={0.6}{stealth}, arrow data={0.9}{stealth}] plot [smooth, tension=.7]
       coordinates {(8.075,0) (6.9,1.5) (6.425,0) (6.9,-1.5) (8.075,0)
                    (9.25,1.5) (9.725,0) (9.25,-1.5) (8.075,0)};

% %%%%%%%%%%%%%%%%%%%%%%%%%%%%%%%%%%%%%%%%%%
\tikzstyle{every node}=[font=\scriptsize]
\node (1) at (0,-0.27) {$-1$};
\node (3) at (1.65,-0.27) {$0$};
\node (5) at (3.3,-0.27) {$t_0$};
\node (7) at (4.95,-0.27) {$\infty$};

\node (8) at (7.25,-0.27) {$-1$};
\node (10) at (8.9,-0.27) {$0$};
\node (12) at (10.55,-0.27) {$t_0$};
\node (13) at (12.1,-0.27) {$\infty$};

\end{tikzpicture}
\caption{Paths on $\PP^1(\CC)$ and ...} \label{fig:paths}
\end{figure}
Next, w define the loops $F$ and $G$ as lifts of this paths to $X_t$
as in Figure~\ref{fig:lifts}
\begin{figure}[ht]
\begin{tikzpicture}[scale=1]
\tikzset{arrow data/.style 2 args={%
      decoration={%
         markings,
         mark=at position #1 with \arrow{#2}},
         postaction=decorate}
      }%

\draw (0,0) node(1){} -- (.825,0) node(2){} -- (1.65,0) node(3){} -- (2.475,0) node(4){} -- (3.3,0)
node(5){} -- (4.125,0) node(6){} -- (4.95,0) node(7){};
\fill (1) circle (1.2pt)
      (2) circle (1.2pt)
      (3) circle (1.2pt)
      (4) circle (1.2pt)
      (5) circle (1.2pt)
      (6) circle (1.2pt)
      (7) circle (1.2pt);

%Figure
\draw [arrow data={0.1}{stealth}, arrow data={0.4}{stealth},
       arrow data={0.6}{stealth}, arrow data={0.9}{stealth}] plot [smooth, tension=.7]
       coordinates {(2.475,0) (1.3,1.5) (.825,0) (1.3,-1.5) (2.475,0)
                    (3.65,1.5) (4.125,0) (3.65,-1.5) (2.475,0)};

% Linien, Punkte
\draw (7.25,0) node(8){} -- (8.075,0) node(9){} -- (8.9,0) node(10){} -- (9.725,0) node(11){} -- (10.55,0)
node(12){} -- (12.1,0) node(13){};
\fill (8) circle (1.2pt)
      (9) circle (1.2pt)
      (10) circle (1.2pt)
      (11) circle (1.2pt)
      (12) circle (1.2pt)
      (13) circle (1.2pt);

%Figure
\draw [arrow data={0.1}{stealth}, arrow data={0.4}{stealth},
       arrow data={0.6}{stealth}, arrow data={0.9}{stealth}] plot [smooth, tension=.7]
       coordinates {(8.075,0) (6.9,1.5) (6.425,0) (6.9,-1.5) (8.075,0)
                    (9.25,1.5) (9.725,0) (9.25,-1.5) (8.075,0)};

% %%%%%%%%%%%%%%%%%%%%%%%%%%%%%%%%%%%%%%%%%%
\tikzstyle{every node}=[font=\scriptsize]
\node (1) at (0,-0.27) {$-1$};
\node (3) at (1.65,-0.27) {$0$};
\node (5) at (3.3,-0.27) {$t_0$};
\node (7) at (4.95,-0.27) {$\infty$};

\node (8) at (7.25,-0.27) {$-1$};
\node (10) at (8.9,-0.27) {$0$};
\node (12) at (10.55,-0.27) {$t_0$};
\node (13) at (12.1,-0.27) {$\infty$};

\node (A) at (1.65,1.5) {$6$};
\node (B) at (.825,-1.5) {$1$};
\node (C) at (4.125,-1.5) {$2$};
\node (D) at (4.125,1.5) {$F$};

\node (E) at (6.4,1.5) {$1$};
\node (F) at (8.6,1.5) {$6$};
\node (G) at (9.725,1.5) {$G$};
\node (H) at (9.725,-1.5) {$2$};
\node (I) at (7.5,-1.5) {$1$};

\end{tikzpicture}
\caption{... their lifts to $X$.} \label{fig:lifts}
\end{figure}
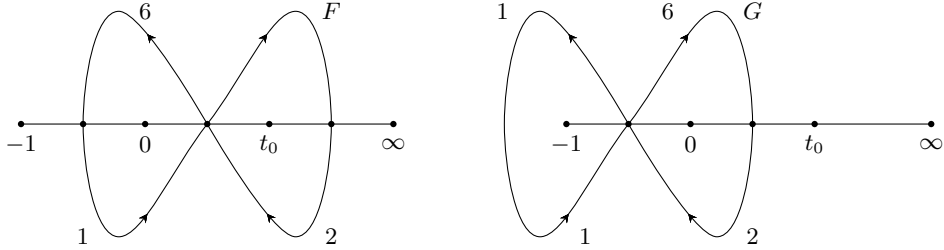
with the convention that the lower left end of the `butterfly'
is always on sheet number one. \footnote{We follow the sheet numbering convention
that clockwise loops around $-1$, $0$ and $t$ are
given by the permutation $\gamma=(123456)$, while the
loop around $\infty$ is given by $\gamma^{-3}= (14)(25)(36)$.
The action of $\alpha$ corresponds to a clockwise turn and thus
decreases the sheet number by one. We also use the convention that
the intersection number is plus one if the intersection is pointing
as the fingers of the right hand.
}
The set of paths $\{ u_{1+k} = \alpha^k(F), u_{7+\ell} = \alpha^\ell(G),\,
k,\ell=0,\ldots,5\}$ has intersection matrix  $M = \langle u_i, u_j \rangle$
given by
$$
{
M= \left( \begin{smallmatrix*}[r]% {rrrrrrrrrrrr} %{cccccccccccc}
0 & \ -1 &0 & 0 & 0 & 1 & \ -1 &  1 & 0  & 0 &0 & 0 \\
1 & 0 &\ -1 & 0 & 0 & 0 &  0 & -1 & 1  & 0 &0 & 0 \\
0 & 1 &0 & \ -1 & 0 & 0 &  0 &  0 & -1 & 1 &0 & 0  \\
0 & 0 &1 & 0 & \ -1 & 0 &  0 &  0 & 0  &-1 &1 & 0   \\
0 & 0 &0 & 1 & 0 & \ -1 &  0 &  0 &  0 & 0  &-1 &1    \\
-1 & 0 &0 &0 & 1 & 0 &  1 &  0 &  0 & 0  &0 &-1    \\
1 & 0 &0 & 0 & 0 & -1 & 0 & \ -1 & 0  & 0 &0 & 1 \\
-1 & 1 &0 & 0 & 0 & 0 &  1 & 0 & \ -1  & 0 &0 & 0 \\
0 & -1 &1 & 0 & 0 & 0 &  0 &  1 & 0 & \ -1 &0 & 0  \\
0 & 0 &-1 & 1 & 0 & 0 &  0 &  0 & 1  &0 &\ -1 & 0   \\
0 & 0 &0 & -1 & 1 & 0 &  0 &  0 &  0 & 1  &0 &\ -1    \\
0 & 0 &0 & 0 & -1 & 1 &  \ -1 &  0 &  0 & 0  &1 &0    \\
\end{smallmatrix*} \right)}\,.
$$
Since the $8\times 8$ minor of the intersection matrix corresponding to $u_1,u_2,u_3,u_4,u_7,u_8,u_9,u_{10}$ is
non-degenerate, it follows that these paths generate $H_1(X,\ZZ)$. We compute that
the homology classes
\begin{equation*}\begin{aligned}
e_1 := u_1,\quad & e_2 := u_3, &\ e_3 := u_1 - u_3 + u_5 + u_6,\quad
& e_4:= u_2 -u_5-u_8\\
e_5 := u_7,\quad & e_6 := u_9, &\ e_7:= u_2+u_3-u_5+u_7, \quad& e_8 := u_1 + u_2 + u_4 + u_6\\
\end{aligned}\end{equation*}
%$$\{e_1,\ldots,e_8\}$
form a symplectic basis for  $H_1(X,\ZZ)$.

To compute the period matrix in this basis, it suffices to determine $f_i = \int_F \omega_i$ and $g_i = \int_G \omega_i$.
since $\alpha$ acts on $\omega_i$ by a power of $\zeta$ and since
\begin{equation} \label{eq:alphaomega}
\int_{\alpha^k(F)} \omega_i = \int_{F} \alpha^{-k}(\omega_i).
\end{equation}
First,  we use the intermediate 3-to-1 covering of $\PP^1$ by the family
of elliptic curves
$\cE_t: y^3 = x(x+1)(x-t)$.
This covering is a constant family of elliptic curves, since there is
no ramification at infinity. We let $q:\cX_t \to \cE_t, (x,y) \mapsto (x,y^2)$
be the quotient map. The holomorphic one-form on $\cE_t$ is $\omega_E = dx/y^2$
and its pullback is $q^* \omega_E = \omega_2$. On $E$ we may suppose $t=1$ and then there is
an additional automorphism $\beta_E: (x,y) \mapsto (-x,\zeta y)$ besides
$\alpha_E: (x,y) \mapsto (x,\zeta^2 y)$. We view $\cE_t$ as a
three-sheeted covering of the projective line and decompose each sheet into two
half-planes, which we number and then label $U$ and $L$ for upper and lower. Then
 $\beta_E$ maps $1U$ to $1L$, maps $1L$ to $3U$, and so on --- consistently with the fact that $\beta_E^2 = \alpha_E$ decreases the
sheet number by one. Now denoting $F_E:=q(F)$ and $G_E:=q(G)$, we have
$\beta_E(G_E) = -\alpha_E^{-1}(F_E)$ and $\beta_E(F_E) = - \alpha_E^2(G_E)$. Hence
$$ \int_{F_E} \omega_E = - \int_{\beta(G_E)} \alpha_E^* \omega_E = - \zeta^{-4} \int_{G_E}
(\beta^{-1})^* \omega_E = \zeta^{-2} \int_{G_E} \omega_E$$
and consequently $g_2 = \int_G q^*\omega_E = \int_{G_E} \omega_E =
\zeta^2 \int_F \omega_2 = \zeta^2 f_2$ gives us the row in the period matrix corresponding to integrals of $\omega_2$

The holomorphic one-form $\omega_1$ is the pullback of the one-form from the
non-isotrivial family
of elliptic curves $y^2=x(x+1)(x-t)$.
We keep~$f_1$ and~$g_1$ as indeterminates and determine the rest
of the first row using~\eqref{eq:alphaomega}.
\par
Finally, $\omega_3$ and $\omega_4$ belong to a unitary local system
and we may calculate their
periods at any special point. We choose $t=1$, where $\cX_1$ has the extra
automorphism $\beta: (x,y) \mapsto (-x,\zeta y)$ with $\beta^2 = \alpha$.
Similarly to the situation for $\cE_t$, we see that $\beta$ maps the
sheets by the pattern $1U\mapsto 1L \mapsto 6U \mapsto 6L \mapsto \cdots$.
Consequently,
$\beta(G) = -\alpha^{-1}(F)$ (and $\beta(F) = - \alpha^2(G)$). We thus get
$$ f_3 = \int_{F} \omega_3 = - \int_{\beta(G)} \alpha^* \omega_3 = - \zeta^{-5} \int_{G}
(\beta^{-1})^* \omega_3 = \zeta^{-5} \int_{G} \omega_3$$
and $f_4 = - \zeta^{-5} g_4$ due to an extra minus sign appearing when pulling
back by $\beta^{-1}$. Another application of the same argument implies that for
our special point $t=1$ we have $f_4 = \zeta^{-3}g_4$.
\par
The previous calculations also imply that we may normalize the
differentials so that $f_i=1$ for $i=1,2,3,4$. We let $\tau = g_1$
and altogether the period matrix with respect to the $\omega$ and
$\{e_1,\ldots,e_8\}$ is as stated in the proposition.
\end{proof}
\par
\medskip
While the automorphism~$\alpha$ of order~$6$ may not deform to any genus four curves
outside the family $X_t$ above (and from our description it would in fact follow
that it does not, but we will not need this result), the automorphism $\varphi = \alpha^2$
of order~$3$ by definition does deform to the Hurwitz space $\cH $. Thus in what follows
the cover $p: X \to E$ given by $(x,y) \mapsto (x,y^3)$ that is obtained as the quotient
by $\varphi$ will be most important to us.
\par
We want to exhibit an isogeny $\Jac(X) \to \Prym(X/E) \times E$.
On the level of one-forms, this is easy, since the first row of
the period matrix $(Z_1\,|\, Z_2)$ corresponds
to $\omega_1$, which is $\varphi$-invariant. Moreover, we identify the period
lattice $\Lambda$ of $\Jac(X) = \CC^4 /\Lambda$ with $\ZZ^8$ and
consider the sublattices $\Lambda_E$ resp.\  $\Lambda_P$ of $\Lambda$ that
are orthogonal to the real and imaginary part of the second through
fourth row (resp.\ first row) of $(Z_1\,|\, Z_2)$. Considering $\Lambda_E$
as a sublattice of $\ZZ^8$ and writing its components as the first and
fifth column of a base change matrix $B$, and proceeding similarly with $\Lambda_P$,
we obtain the matrix
\be \label{eq:STbasechange}
B := \left( \begin{smallmatrix*}[r]
 1 &\ \ 0 &\ -1 &\ -1 & 1 &\ \ 0 &\ \ 1 &-2\\
 1 &0 & 1 & 2 & 1 &0 &0 & 1\\
 0 &1 & 0 & 0&  0& 0& 0&  0\\
-1 &0 & 0 & 1& \ -1& 0& 1& \ -1\\
 0 &0 & 0 & 1&  1& 0& 0&  0\\
 0 &0 & 0 & 1&  1& 0& 1&  0\\
 0 &0 & 0 & 0&  0& 1& 0&  0\\
 1 &0 & 0&  1&  0& 0& 0&  1\\
 \end{smallmatrix*} \right)
\ee
such that
\be \label{eq:Z1Z2B}
(Z_1\,|\, Z_2) B  = \left(\smallmatrix 3\tau & 0 &\ 3\tau +3 & 0 \\
0 &\ Z_1^{(3)} & 0 &\ Z_2^{(3)} \endsmallmatrix\right)
\ee
is indeed the period matrix of the product of $E$ and the period matrix of $\Prym(X/E)$, which has period matrix
$(Z_1^{(3)}\,|\, Z_2^{(3)})$ given by
\begin{equation}\label{Z3special}
\left( \begin{smallmatrix*}[r]
-\zeta^2 + 1 &\ -\zeta^2 + 2 & -3\zeta^2 + 6 & \zeta^2 & 0& -3\zeta^2 + 3 \\
-2\zeta^3 + 2\zeta^2 + \zeta - 1 & 2\zeta^3 + \zeta &\ -3\zeta^3 + 3\zeta^2 &
\phantom{-} \zeta^3 - \zeta^2 - 2\zeta + 2 & \phantom{-}\zeta^3 + 2\zeta^2 - 2\zeta - 1 & \ -3\zeta^3 + 3\zeta \\
\phantom{-}2\zeta^3 + 2\zeta^2 - \zeta - 1 & \phantom{-}2\zeta^3 - \zeta &
\phantom{-} 3\zeta^3 + 3\zeta^2 &\  -\zeta^3 - \zeta^2
+ 2\zeta + 2 &\ -\zeta^3 + 2\zeta^2 + 2\zeta - 1 &\phantom{-}3\zeta^3 - 3\zeta \\
\end{smallmatrix*} \right)\,.
\end{equation}
We check that the polarization $J_3 = B^T J B $, where $J$ is the standard principal polarization matrix with ones
on the diagonal of the upper right block, is indeed of type $(1,1,3)$.
Furthermore, we see that the $(1,1,3)$ polarized abelian threefold with the period matrix  $(Z_1^{(3)}\,|\, Z_2^{(3)})$ indeed admits indeed an order three automorphism, preserving the polarization $J_3$, given by the diagonal matrix ${\rm diag}(\zeta^4,\zeta^8,\zeta^8)$ acting on
the left and the inverse of
\be \label{eq:M3}
M_3 := \left( \begin{smallmatrix*}[r]
0 &0 &0 &-1& 0& 0\\
0 &1 &0 &0& -3& 3\\
0 &0& 1& 0& 1& 0\\
1 &0 &0 &\ -1 &0 &0\\
0 &0 &-3& 0 &\ -2 &0\\
0 &\ -1 &\ -3 &0 &0&\ -2\\
\end{smallmatrix*} \right)
\ee
acting on the right.

%%%%%%%%%%%%%%%%%%%%%%%%%%%%%%%%%%%%%%%%%%%%%%%%%%%5
\section{Families of abelian threefolds with complex multiplication}
\label{sec:A3CM}
%%%%%%%%%%%%%%%%%%%%%%%%%%%%%%%%%%%%%%%%%%%%%%%%%%%5

The moduli spaces of abelian varieties with given endomorphism ring
and polarization, nowadays called PEL-Shimura varieties, have been
constructed by Shimura in \cite{shimura}. This construction is presented in
the textbook \cite[Chapter~9]{bl}, the notation of which we follow, and
to which we refer for further details.
\par
In the previous section we dealt with the Shimura-Teichm\"uller curve~$\cX$
the general Jacobian in which has a $\ZZ/6$ automorphism. Our goal now is to
explicitly construct $(1,1,3)$-polarized abelian threefolds with an order
three automorphism inducing complex multiplication by $\ZZ[\rho]$ with
$\rho:=\eta^4$ a third root of unity. The moduli space of polarized
abelian threefolds with such an endomorphism is two-dimensional.
\par
\begin{Prop}\label{prop:genus3}
There exists an irreducible component of the moduli space of $(1,1,3)$
polarized abelian threefolds with a $\ZZ/3$ automorphism such that
the $3\times 6$ period matrices are $(Z_1^{(3)}(z_1, z_2) \,|\, Z_2^{(3)}(z_1, z_2)$
for a suitable choice of basis, where the matrix $Z_1^{(3)}$ is equal to
\bas
\left(\begin{smallmatrix*}[r]
%\sqrt[4]{\tfrac13}\,
3^{-1/4}(\zeta^2+3\zeta+1)\, z_2 & (\zeta^2+3\zeta+1)(z_1+1)
&\ \ (-\zeta^3+8\zeta+3)(z_1+1) \\
-2\zeta^3 +2\zeta^2 + \zeta -1 &3^{3/4}\,(\zeta^3+\zeta^2-1) \,z_2 -3\zeta^2(z_1-1)
& a_{23}\\
2\zeta^3 +  2\zeta^2 -\zeta -1 &\ \ \ 3^{3/4}\,(\zeta^3+\zeta^2-1) \,z_2 +
(4\zeta^3 + 2\zeta^2 -5\zeta -4)(z_1-1)
& a_{33}\\
\end{smallmatrix*} \right)
\eas
where
$$
\scriptstyle
a_{23}\=3^{3/4} (3\zeta^3 + 3\zeta^2-2\zeta -1) \,z_2 - 3(\zeta^2-3\zeta-1) (z_1 -1)
$$
and
$$
\scriptstyle
a_{33}\=
3^{3/4}\,(3\zeta^3+3\zeta^2-\zeta-4) \,z_2 + (10\zeta^3 + \zeta^2 -17\zeta -11)
\,z_1  + (8\zeta^3 +2\zeta^2 -13\zeta -11);$$
while the matrix $Z_2^{(3)}$ is equal to
\bas
\left(\begin{smallmatrix*}[r]
3^{-1/4}(3\zeta^3 + \zeta^2 - 3\zeta -2) \,z_2&\ \  (3\zeta^3 + \zeta^2 - 3\zeta -2)
(z_1+ 1)  & \ \ b_{13} \\
\zeta^3 -\zeta^2 -2\zeta +2 & \ \  3(z_1+1) + 3^{3/4}(\zeta^3 -\zeta +1)\,z_2
&   b_{23} \\
-\zeta^3 -\zeta^2 -2\zeta +2 & \ \ 3^{3/4}(-\zeta^3 +\zeta +1)\,z_2
+ (\zeta^3 + 2\zeta^2 + 4\zeta + 2 )(z_1+1)
& b_{33} \\
\end{smallmatrix*} \right)
\eas
where
\bas \scriptstyle
b_{13} & \=  \scriptstyle (8\zeta^3 + 3\zeta^2 -7\zeta -3) \,z_1
\+  10\zeta^3 + 6\zeta^2 -8\zeta -6\\
\scriptstyle b_{23} &\=  \scriptstyle (9\zeta^3 - 3)\,z_1
\,-\, 3^{3/4}(12\zeta^3 -3\zeta^2 -9\zeta +9)\,z_2  \+ 9\zeta^3 -3\zeta^2 \\
\scriptstyle b_{33} &\=  \scriptstyle (\zeta^3 + 10\zeta^2 + 10\zeta
+ 7)\,z_1 \,-\, 3^{3/4}(\zeta^3 - \zeta^2 -3\zeta -3)\,z_2 \+ 5\zeta^3 + 5\zeta^2+8\zeta +2
\eas
for $(z_1,z_2)$ in some complex $2$-ball.
Moreover, the $3\times 6$ period matrix obtained for
$$z_1 \= \tfrac{-2\zeta^3 + \zeta^2 + \zeta -3}2, \qquad z_2 \= 3^{-1/4}
\, \tfrac{\zeta^3 -2\zeta^2 +1}2.$$
is precisely equal to the matrix \eqref{Z3special} obtained in the previous
section.
\end{Prop}
Once we have proven this proposition, our main result follows immediately
\begin{proof}[Proof of Theorem \ref{thm:main}]
Indeed, once we have determined the period
matrices $(Z_1^{(3)}(z_1, z_2) \,|\,
Z_2^{(3)}(z_1, z_2))$ of the Pryms appearing in our construction, to obtain the period matrices of the genus four Jacobians that
form Shimura curves, we will substitute them into the form of period matrices
given by \eqref{eq:Z1Z2B} and undo the base change by the matrix $B$ given by \eqref{eq:STbasechange}.
\end{proof}
\par
\medskip
Thus it remains to prove the proposition using Shimura's construction.
However, the construction of universal families by Shimura involves
quite a number of choices. So this section is a guide to the practical
use of Shimura's construction: how to adapt the choices so that a given
period matrix appears as a member of the family.
\par
\begin{proof}[Proof of Proposition \ref{prop:genus3}]
For the proof, we follow Shimura's original construction, adapted and specialized to our case,
as explained in \cite[Chapter 9]{bl}, to which we refer for all justifications of the steps.
In our case the endomorphism ring is the maximal order $\ZZ[\rho]$
in the field $K = \QQ(\rho)$, while the Rosati involution is simply the
complex conjugation, and its fixed is $K_0=\QQ$, of which $K$ is thus a quadratic extension.
Abelian threefolds with this endomorphism ring are given as quotients of the  $2$-ball
$$ B_2 := \{ z = \left(\begin{smallmatrix} z_1 \\ z_2 \end{smallmatrix} \right)
 \in \CC^2: |z_1|^2 + |z_2|^2 < 1 \}.$$
Following Shimura, to construct {\em all} such abelian threefolds (for complex multiplication
by {\em any} order in~$K$) one chooses any pair $(\cM,T)$ consisting of a free rank~$6$ submodule $\cM$
of $K^3$ and a non-degenerate signature $(2,1)$ matrix $T \in \operatorname{Mat}_{3\times 3}(K)$ such that
$\overline{T}^T = -T$, satisfying
$$ {\rm tr}^K_\QQ(a^T T \overline{b}) \in \ZZ \quad \text{for all} \quad a,b \in \cM\,.$$
To construct such $(\cM,T)$, recall that the signature (2,1) condition means that
there exists a matrix $W \in \GL_3(\CC)$ satisfying
\be \label{eq:defW}
 T \= W^T \mat {I_2} 00{-i} \overline{W}\,,
\ee
where we denote by $I_2$ the product of $i$ and the $2\times 2$ identity matrix.
To any $z \in B_2$ we associate the linear map $J_z: \CC^6 \to \CC^3$
given by left multiplication by the matrix
$$ J_z := \mat{(z^t,1)W} 0 0 {(I_2,z)\overline{W}}\,$$
and we let $j: \cM \to \CC^6$ be given by one embedding of $K \to \CC$ in the
first three coordinates and the complex conjugate embedding in the remaining
three coordinates. Then the abelian variety $X_z := \CC^3/J_z(j(\cM))$ with the
polarization
$$ H := \left(\begin{smallmatrix}
|z|^{-1} & 0 \\ 0 & (I_2 - \ol{z} z^T)^{-1} \\
\end{smallmatrix}\right) $$
has the desired endomorphism, since by a direct matrix computation one checks that
$$ a J_z(j(b)) \= J_z(j(ab)) \quad \text{for all} \quad a \in K\, \quad\text{and}\ b\in\cM\, ,$$
and hence the action of $K$ on $\CC^3$ is compatible with the lattice $\Lambda= J_z(j(\cM))$ (see \cite[Chapter~9.3]{bl}).
Moreover, $\operatorname{Im} H$ is indeed an integer-valued bilinear form on $\Lambda$, as one checks by computing
\be \label{eq:polidentity}
(\operatorname{Im} H)\left(J_z(j(a)),J_z(j(b))\right) \ = {\rm tr}^K_\QQ(a^T T \overline{b})
\ee
for all $a,b \in \cM$.

\medskip
Consequently, our goal is to find a pair $(\cM,T)$ such that the period matrix
$(Z_1^{(3)}\,|\, Z_2^{(3)})$ of the Shimura-Teichm\"uller curve, given by \eqref{Z3special}, appears in the family of period matrices $J_z(j(\cM))$
for some $(z_1, z_2) \in B_2$. To find $\cM$, note that
since $(Z_1^{(3)}\,|\, Z_2^{(3)})$ has complex multiplication by a maximal order, we must have $\cM \cong \ZZ[\rho]^3$.
Hence we have to pick three elements in $\ZZ^6$ that together with their
$\rho$-images (with $\rho$ acting by left multiplication by the matrix $M_3$ given
in~\eqref{eq:M3}, computed for the special case) generate $\ZZ^6$. A choice of such elements
is given for example by $u_1 = e_1$, $u_2 = e_2 + e_5$, $u_3 = 2e_3 + e_5 -e_6$, and $u_{3+k} = \rho u_k$
for $k=1,2,3$.

Let now $L$ be the base change matrix from $\ZZ^6$ to the basis  $u_1,\ldots,u_6$. We compute
$$ L^T J_3 L \=  \left(\begin{smallmatrix*}[r]
0 & 0 &0 \\ 0 & 0 & 1 \\ 0 &\ -1 &\ \ 0 \\
\end{smallmatrix*} \right), \quad L^T M_3^T J_3 L \ =  \left(\begin{smallmatrix*}[r]
-1 & 0 & 0 \\ 0& 0& 1 \\ 0 &\ \ 2&\ \ 9 \\
\end{smallmatrix*} \right).
$$
Now to make a choice of $T$, we use \eqref{eq:polidentity} and the fact that $\operatorname{Im} H$ is the intersection form on $\Lambda$. The matrix
$$T \= \left(\begin{smallmatrix*}[r]
\tfrac23 \zeta^4 + \tfrac1 3 & 0 & 0\\ 0 &0& -\zeta^4 \\ 0 &\ \ -\zeta^4-1
&\ \ -6\zeta^4 - 3 \\
\end{smallmatrix*} \right)
$$
has the desired property~\eqref{eq:polidentity}, since $\tr^K_\QQ(T + \ol{T})
= L^T J_3 L$ and since $\tr^K_\QQ(\zeta^4T + \ol{\zeta^4T}) = L^T M_3^T J_3 L$.
We can now find a $W$ satisfying \eqref{eq:defW}, and we choose
$$ W \= \left(\begin{smallmatrix*}[r]
0 & \ \ 3 - \zeta  & 0 \\
3^{-1/4} & 0 & 0 \\
0 & 1 & \ \ 3-i \\
\end{smallmatrix*} \right).
$$
Substituting all of these choices, for Shimura's form of the period matrices we finally obtain
\bas
(Z_1^{(S)}(z_1,z_2)\,|\, Z_2^{(S)}(z_1,z_2))  = \qquad\qquad\qquad\qquad\qquad\qquad \\
\left(\begin{smallmatrix*}[r]
3^{-1/4}\,z_2\, &\ z_1 + 1\, & (3-\zeta)\,z_1 + 3-i\, &
\ 3^{-1/4}\;\zeta^4\, z_2\, &\, \zeta^4(z_1 + 1)\, & ((3-\zeta)\,z_1 + 3-i)\zeta^4\\
0 \,& z_1 + 1\,&\  (3-i)(z_1+1) + 3-\zeta^{-1}&
0 \,& (z_1 + 1)\zeta^8 \, &\,\  ((3-i)(z_1+1) + (3-\zeta^{-1})) \zeta^8\\
3^{-1/4}\,& z_2 \,& (3+i)\,z_2 \, &
3^{-1/4}\;  \zeta^8\,&  \zeta^8\,z_2 \,& (3+i) \zeta^8\,z_2 \\
\end{smallmatrix*} \right)
\eas
\par
Finally, the last indeterminacy in the choice of the period matrices is
that we can choose different eigenforms within the two eigenspaces. This
means we can further have a base change matrix $C = (c_{ij})$ for holomorphic
one-forms and are looking for $z=(z_1,z_2)$ such that
\be \label{eq:solveforspec}
 \left(\begin{smallmatrix}
c_{11} & 0  & 0 \\
0 & c_{22} & c_{23} \\
0 & c_{32} & c_{23} \\
\end{smallmatrix} \right) (Z_1^{(S)}(z)\,|\ Z_2^{(S)}(z)) \= (Z_1^{(3)}\,|\, Z_2^{(3)}) \,.
\ee
This system has a unique solution, given by
\par
\bas \scriptstyle \begin{matrix*}[l]
\scriptstyle c_{11} \= \zeta^2 + 3\zeta + 1, &\scriptstyle
\ c_{22} \= -3\zeta,&\scriptstyle
c_{23} \= 3^{3/4}\,(\zeta^3-\zeta^2+1), \\
& \scriptstyle\ c_{32} \=  4\zeta^3 +2\zeta^2 -5\zeta -4, &\scriptstyle
c_{33} \= 3^{3/4}\,(\zeta^3+\zeta^2-1)
\end{matrix*}
\eas
and $(z_1,z_2)$ as stated in the proposition.
With these values of $c_{ij}$, the left hand side of~\eqref{eq:solveforspec}
is the family given in the statement of the proposition.
\end{proof}

%\bibliographystyle{halpha}
%\bibliography{my}

\end{document}